\documentclass[12pt]{amsart}
\usepackage[margin=1in]{geometry} 
\usepackage{graphicx, appendix}
\usepackage[dvipsnames]{xcolor}

\usepackage{tikz}
\usetikzlibrary{shapes}

\newtheorem{theorem}{Theorem}[section]
\newtheorem{lemma}[theorem]{Lemma}
\newtheorem{prop}[theorem]{Proposition}
\newtheorem{cor}[theorem]{Corollary}
\theoremstyle{definition}

\newcommand{\N}{\mathbb{N}}

\title{The distribution of the number of parts of $m$-ary partitions modulo $m$.}

\author{Tom Edgar}

\address{Department of Mathematics, Pacific Lutheran University, Tacoma, WA 98447}
\email{edgartj@plu.edu}


\date{\today}

\begin{document}

\begin{abstract}
We investigate the number of parts modulo $m$ of $m$-ary partitions of a positive integer $n$. We prove that the number of parts is equidistributed modulo $m$ on a special subset of $m$-ary partitions. As consequences, we explain when the number of parts is equidistributed modulo $m$ on the entire set of partitions, and we provide an alternate proof of a recent result of Andrews, Fraenkel, and Sellers about the number of $m$-ary partitions modulo $m$.
\end{abstract}

\maketitle

\section{Preliminaries and Statement of the Main Result}

Throughout this note, we let $\N=\{0,1,2,3,...\}$ represent the set of natural numbers. For any $m\geq 2$, every natural number $n$ has a unique base-$m$ representation of the form $n = n_0 + n_1m + \cdots + n_k m^k$ with $n_k\ne 0$.  We express this more compactly as $n=(n_0, n_1, \ldots, n_k)_m$ and use the convention that $n_i=0$ if $i>k$.

For $m\geq 2$, we say a partition of $n\in\N$ is an $m$-ary partition if each part is a power of $m$. We let $b_m(n)$ represent the number of $m$-ary partitions of $n$. For instance, the $2$-ary partitions of $8$ are
$$8,\ \ 4+4,\ \ 4+2+2,\ \ 4+2+ 1+ 1,\ \ 4+1+1+1+1,\ \  2+ 2+ 2+ 2, 2+ 2+ 2+ 1+ 1,$$
$$\ \  2+ 2+ 1+ 1+ 1+1,\ \ 2+ 1+ 1+ 1+1+ 1+ 1, \ \ 1+ 1+ 1+ 1+ 1+ 1+ 1+ 1,$$
so that $b_2(8)=10$.

In a recent issue of \textit{The American Mathematical Monthly}, Andrews, Fraenkel, and Sellers (see \cite{EFS}) provided the following beautiful characterization of the number of $m$-ary partitions mod $m$ relying only on the base-$m$ representation of a number. 
\begin{theorem}[Andrews, Fraenkel, Sellers]
\label{afs}
If $m\geq 2$ and $n=(n_0,n_1,\ldots,n_k)_m$ then
$$b_m(mn)=\prod_{i=0}^k(n_i+1) \pmod{m}.$$
\end{theorem}

Their elegant proof follows from clever manipulation of power series and the generating function for $m$-ary partitions. Their result allows for a uniform proof of many known congruence properties of $m$-ary partitions originally conjectured by Churchhouse and proved by R\o dseth, Andrews and Gupta (see \cite{churchhouse}, \cite{rodseth}, \cite{andrews}, \cite{gupta1}, and \cite{gupta2}). 

Theorem \ref{afs} implies that 
$$b_m(mn) -  \prod_{i=0}^k(n_i+1) =  m\cdot q$$
for some $q\in\N$. Our primary result (Theorem \ref{equidistributethm}) provides a combinatorial interpretation for the value of $q$. Furthermore, as a corollary to our main result, we obtain a new proof of Theorem \ref{afs} that does not rely on generating functions. 

We note that the product in Theorem \ref{afs}, $\prod_{i=0}^k(n_i+1)$, arises in various other places; for instance, when $m$ is prime, this number counts the nonzero entries in row $n$ of Pascal's Triangle mod $m$ (see \cite{fine}). This product can also be interpreted in terms of a partial order on the Natural numbers arising from base-$m$ representations. In particular, for fixed $m\geq 2$, we let $\ll_m$ represent the digital dominance order defined by $a\ll_mb$ if $a_i\leq b_i$ for all $i$, where $a=(a_0,a_1,\ldots, a_k)_m$ and $b=(b_0,b_1,\ldots,b_l)_m$ (see \cite{edgarjudaball} or \cite{judaball}). Then, for $n=(n_0,n_1,\ldots,n_k)_m$, the product $\prod_{i=0}^k(n_i+1)$ counts the number of integers dominated by $n$ (see \cite{edgarjudaball}). We will use the interpretation of the product in terms of the $m$-dominance order in what follows.

Now, let $n$ be a positive integer with $m^{k}\leq n< m^{k+1}$, then every $m$-ary partition is of the form
$$\ell_k\cdot m^k+\ell_{k-1}\cdot m^{k-1}+\cdots \ell_1\cdot m+\ell_0$$
with $\ell_i\geq 0$ for all $i$. We will denote such a partition by $[\ell_0,\ell_1,\ldots,\ell_{k-1},\ell_{k}]_m$. We mention here that the base-$m$ representation of $n$ yields an $m$-ary partition 
$$(n_0,n_1,\ldots, n_k)_m\mapsto [n_0,n_1,\ldots, n_k]_m.$$

Finally, We define a function $nops$ from $m$-ary partitions of $n$ to $\N$ by 
$$nops([\ell_0,\ell_1,\ldots,\ell_{k-1},\ell_{k}]_m)=\sum_{i=0}^k\ell_i;$$ 
this represents the number of parts of the partition.

Now, let $n=(n_0,n_1,\ldots, n_k)_m$. We call an $m$-ary partition, $\ell$, of $n$ \textit{simple} if $\ell=[\ell_0,\ell_1,\ldots,\ell_k]_m$ with $\ell_i\leq n_i$ for all $i\geq 1$. Thus, simple partitions are obtained by replacing powers of $m$ in the $m$-ary representation with the appropriate number of $1$'s. Let $P_m(n)$ be the set of $m$-ary partitions of $n$, $S_m(n)$ be the set of simple $m$-ary partitions of $n$, and $N_m(n)=P_m(n)\setminus S_m(n)$ be the set of non-simple $m$-ary partitions of $n$. Restricting the function $nops$ to $N_m(n)$, we get the following result.

\begin{theorem}
\label{equidistributethm}
Let $m\geq 2$ and $n\in\N$. Then the $nops$ function is equidistributed modulo $m$ on the set $N_m(n)$.
\end{theorem}

As a corollary, we obtain the following.

\begin{cor}
\label{congruence}
Let $b_m(n)$ be the number of $m$-ary partitions of $n=(n_0,n_1,\ldots,n_k)_m$, then 
$$b_m(n)\equiv \prod_{i=1}^k(n_i+1) \pmod{m}.$$
\end{cor}

We note that the previous corollary is stated slightly differently than Theorem \ref{afs}, which is given only for $b_m(mn)$; however, due to the fact that $b_m(mn+r)=b_m(mn)$ when $0<r<m$ (as stated in \cite{EFS}), the two forms are equivalent.

This paper is organized as follows. Section \ref{technical} contains the details necessary to prove Theorem \ref{equidistributethm}. We prove the theorem and its corollary in Section \ref{proofofthm}. In addition, we use Theorem \ref{equidistributethm} to describe when $nops$ is equidistributed mod $m$ on the entire set of $m$-ary partitions, $P_m(n)$. Section \ref{partitionexample} contains a detailed example illustrating the results in Sections \ref{technical} and \ref{proofofthm}. Finally, in Section \ref{future} we describe some possible extensions.
\section{Technical details}
\label{technical}

In this section, we provide a systematic way to partition $N_m(n)$, which will be used to prove Theorem \ref{equidistributethm}. We have included a detailed example of this method of partitioning in Section \ref{partitionexample}.

Let $m\geq 2$ and $n\in \N$ be fixed with $n=(n_0,n_1,\ldots,n_k)_m$. We first define a function $f_{m,n}:N_m(n)\to \N$ by
$$f_{m,n}( [\ell_0,\ell_1,\ldots,\ell_k]_m) = (b_0,b_1,b_2,\ldots,b_k)_m,$$
where $b_i=\min(n_i,\ell_i)$ for all $i$; we note that $b_0=n_0$ since $\ell_0\equiv n_0\pmod{m}$. The following lemma follows by construction.

\begin{lemma}
For any non-simple partition $\ell\in N_m(n)$, then we have $f_{m,n}(\ell)\ll_m n$.
\end{lemma}

Now, we use $f_{m,n}$ to define a relation on $N_m(n)$ by $\rho\sim \gamma$ if $f_{m,n}(\rho)=f_{m,n}(\gamma)$.

\begin{lemma}
\label{fpartition}
The relation $\sim$ is an equivalence relation, and so 
$$\{f_{m,n}^{-1}(b) \mid b\in \N \mbox{ and } b\ll_m n \mbox{ and } f_{m,n}^{-1}(b)\ne\emptyset\}$$
forms a partition of $N_m(n)$.
\end{lemma}

\begin{proof}
Any function yields such an equivalence relation.
\end{proof}

\begin{lemma}
\label{decompose}
Let $\ell$ be a non-simple $m$-ary partition of $n$. Then $\ell$ can be component-wise decomposed as
$$\ell = [\ell_0,\ell_1,\ldots,\ell_k]_m=[r_0,r_1,\ldots,r_k]_{m}+[b_0,b_1,b_2,\ldots,b_k]_m$$
where $b=(b_0,b_1,b_2,\ldots,b_k)_m=f_{m,n}(\ell)$ and $r_i\geq 0$ for all $i$. Moreover, it follows that $r_i>0$ only if $n_i=b_i$.
\end{lemma}

\begin{proof}
Since $r_i=\ell_i-\min(\ell_i,n_i)$, it is clear that $r_i\geq 0$. Now if $r_i>0$, then $\min(\ell_i,n_i)\ne \ell_i$ so that $b_i=n_i$ as required.
\end{proof}

\begin{lemma}
\label{maxchain}
Let $\ell$ be a non-simple $m$-ary partition of $n=(n_0,n_1,\ldots,n_k)_m$ with $\ell\in f_{m,n}^{-1}(b)$ where $b=(b_0,b_1,...,b_k)_m$. Suppose that $\ell$ is of the form 
$$\ell=[\ell_0,b_1,b_2,\ldots, b_{j-1}, \ell_j,\ell_{j+1},\ldots,\ell_k]_m$$
with $\ell_j>n_j=b_j$. Then, there is a unique pair $(r,h)$ with $r\geq 1$ and $0\leq h< m^j$ such that $\ell_j\leq n_j+mr$, there is an $m$-ary partition of the form $[h,b_1,b_2,\ldots, b_{j-1}, b_j+mr,\ell_{j+1},\ldots,\ell_k]_m$, and there is no $m$-ary partition of the form $[h',b_1,b_2,\ldots, b_{j-1}, g,\ell_{j+1},\ldots,\ell_k]_m$ with $g>b_j+mr$.
\end{lemma}

\begin{proof}
Let $s=\ell_j-b_j=\ell_j-n_j> 0$. According to the division algorithm, there is a unique $h$ satisfying $\ell_0=t\cdot m^j +h$ where $0\leq h<m^j$. Then, clearly
$$[h,b_1,b_2,\ldots, b_{j-1}, b_j+s+t,\ell_{j+1},\ldots,\ell_k]_m$$
is an $m$-ary partition of $n$. We then note that 
$$[h',0,0,\ldots, 0, b_j+s+t,\ell_{j+1},\ldots,\ell_k]_m$$
is an $m$-ary partition of $n$ where $h':=h+\sum_{i=1}^{j-1}b_i=\sum_{i=0}^{j-1}n_i<m^j$. This implies that 
$$[0,0,0,\ldots, 0, b_j+s+t,\ell_{j+1},\ldots,\ell_k]_m$$
is an $m$-ary partition of $n'=(0,0,\ldots,0, n_j,n_{j+1},\ldots,n_k)_m$. However, since $n_j=b_j$ and $s+t>0$, then $0<s+t=\sum_{i=j+1}^k(n_i-\ell_i)\cdot m^{i-j}$. Thus, $s+t=mr$ for some $r\geq 1$ as required. Finally, we see that $b_j+mr$ is the largest number of parts of the form $m^j$ we can have without reducing some $\ell_i$ with $i>j$.

\end{proof}

\begin{cor}
\label{maxchain2}
Let $\ell$ be a non-simple $m$-ary partition of $n=(n_0,n_1,\ldots,n_k)_m$ with $\ell\in f_{m,n}^{-1}(b)$ where $b=(b_0,b_1,...,b_k)_m$. Suppose that $\ell$ is of the form 
$$\ell=[\ell_0,b_1,b_2,\ldots, b_{j-1}, \ell_j,\ell_{j+1},\ldots,\ell_k]_m$$
with $\ell_j>n_j=b_j$. Then there is an $m$-ary partition of the form $[v,b_1,b_2,\ldots, b_{j-1}, u,\ell_{j+1},\ldots,\ell_k]_m$ for all $b_j< u\leq b_j+mr$ where $r$ is given by Lemma \ref{maxchain}.
\end{cor}

\begin{proof}
Let $b_j<u\leq b_j+mr$ and consider the partition of the form $\rho=[h,b_1,b_2,\ldots, b_{j-1}, b_j+mr,\ell_{j+1},\ldots,\ell_k]_m$ guaranteed by Lemma \ref{maxchain}. We then find $y$ such that $b_j+mr=u + y$ where $y\geq 0$. Then we construct an $m$-ary partition from $\rho$ by converting $y$ parts of the form $m^j$ to $y\cdot m^j$ parts of the form $m^0$, obtaining the partition
$$[h+y\cdot m^j,b_1,b_2,\ldots, b_{j-1}, u,\ell_{j+1},\ldots,\ell_k]_m$$
as required.
\end{proof}

Now, fix $b\ll_m n$ with $f_{m,n}^{-1}(b)\ne\emptyset$. For each $1\leq z\leq k$, we define 
$$B(z):=\{\rho\in f_{m,n}^{-1}(b)\mid \min\{i\geq 1\mid \rho_i\ne b_i\}=z\}.$$
Again, the following lemma is clear by construction.

\begin{lemma}
\label{jbpartition}
Let $b\ll_m n$ with $f_{m,n}^{-1}(b)\ne\emptyset$. The the collection of sets $\{B(z)\mid B(z)\ne\emptyset\}$ forms a partition of $f_{m,n}^{-1}(b)$.
\end{lemma}

As our final step, we fix $z$ with $1\leq z\leq k$ such that $B(z)\ne\emptyset$. Now, we define a relation on $B(z)$ as follows. We say $\rho\simeq_{b,z} \gamma$ if $\gamma_i=\rho_i$ for all $i>z$.

\begin{lemma}
\label{simeqpartition}
The relation $\simeq_{b,z}$ on $B(z)$ is an equivalence relation and so provides a partition of $B(z)$.
\end{lemma}
\begin{proof}
This is again clear by construction.
\end{proof}

\begin{prop}
\label{smallclass}
Let $n\in \N$, $b\in \N$ with $b\ll_m n$ and $1\leq z\leq k$ such that $f_{m,n}^{-1}(b)\ne\emptyset$ and $B(z)\ne\emptyset$. Then the $nops$ function is equidistributed modulo $m$ on each equivalence class of $\simeq_{b,z}$.
\end{prop}

\begin{proof}
Suppose the $C$ is an equivalence class of $\simeq_{b,z}$. Then by construction, there exists $\ell_{z+1},\ell_{z+2},\ldots, \ell_{k}$ such that every partition in $C$ is of the form 
$$[h,b_1,b_2,\ldots, b_{z-1}, h',\ell_{z+1},\ell_{z+2},\ldots, \ell_k]_m$$
for some $h$ and $h'$ with $h'>b_z$. Now according to Lemma \ref{maxchain2} and Corollary \ref{maxchain2}, there exists some $r\geq 1$ such that
$$C=\{[h,b_1,b_2,\ldots, b_{z-1}, u,\ell_{z+1},\ell_{z+2},\ldots, \ell_k]_m\mid h\in\N \mbox{ and } b_j<u\leq b_j+mr\}.$$
Thus $|C|=mr$. Now, for each $1\leq w\leq m$ we define
$$C_w=\{[h_j,b_1,b_2,\ldots, b_j+w+jm,\ell_{z+1},\ell_{z+2},\ldots, \ell_k]_m\mid 1\leq j\leq (r-1)\},$$
and we note that  $|C_w|=r-1$ for all $w$ and the set $\{C_w\}$ forms a partition of $C$. Moreover, for each $w$, $nops(\gamma)\equiv nops(\rho) \pmod{m}$ for all $\gamma,\rho\in C_w$, and $nops(\rho)\equiv nops(\gamma)+1 \pmod{m}$ whenever $\gamma\in C_w$ and $\rho\in C_{w+1}$.
\end{proof}

\section{Proof of Theorem \ref{equidistributethm} and Consequences}
\label{proofofthm}

\begin{proof}[Proof of Theorem \ref{equidistributethm}]
Let $b\ll_n n$ with $f_{m,n}^{-1}(b)\ne\emptyset$. Then, let $1\leq z\leq k$ with $B(z)$ be non-empty. By Proposition \ref{smallclass} and Lemma \ref{simeqpartition} , the $nops$ function is equidistributed mod $m$ on $B(z)$. Likewise, by Lemma \ref{jbpartition}, the $nops$ function is equidistributed mod $m$ on $f_{m,n}^{-1}(b)$. Finally, Lemma \ref{fpartition} implies that the $nops$ function is equidistributed mod $m$ on $N_m(n)$.
\end{proof}

Let $n=(n_0,n_1,\ldots,n_k)_m$. Then, according to Theorem \ref{equidistributethm}, $N_m(n)  = m\cdot q$ where $q$ is the number of non-simple $m$-ary partitions with with number of parts divisible by $m$. However, it is clear that there is a bijection between simple $m$-ary partitions of $n$ and the integers equivalent to $n$ mod $m$ that are $m$-dominated by $n$:
$$[\ell_0,b_1,b_2,\ldots,b_k]_m\longleftrightarrow (n_0, b_1,b_2,\ldots, b_k)_m.$$
As previously mentioned, there are $\prod_{i=1}^k(n_i+1)$ integers equivalent to $n$ mod $m$ that are $m$-dominated by $n$ (see \cite{edgarjudaball} and use the fact that $b$ is equivalent $n$ mod $m$ if and only if $b_0=n_0$). Thus, we see that 
$$b_m(n) = |N_m(n)|+|S_m(n)| = m\cdot q +\prod_{i=1}^k(n_i+1)$$ 
so that Corollary \ref{congruence} holds.

Understanding the $nops$ function on $N_m(n)$ allows us to characterize when the $nops$ function is equidistributed mod $m$ on the entire set of $m$-ary partitions, $P_m(n)$.

\begin{cor}
\label{nopsonsimple}
The $nops$ function is equidistributed modulo $m$ on $P_m(n)$ if and only if $nops$ is equidistributed modulo $m$ on the simple $m$-ary partitions, $S_m(n)$.
\end{cor}

\begin{proof}
This follows from Theorem \ref{equidistributethm} since $P_m(n)$ is the disjoint union of $N_m(n)$ and $S_m(n)$.
\end{proof}

\begin{theorem}
Let $m\geq 2$ and let $n=(n_0,n_1,\ldots,n_k)_m$ be the base-$m$ representation of $n$. Then the $nops$ function is equidistributed modulo $m$ on $P_m(n)$ if and only if the set $\{n_1,n_2,\ldots,n_k\}$ contains $m-1$.
\end{theorem}

\begin{proof}
First, suppose that $n_i=m-1$ for some $i\geq 1$. Due to Corollary \ref{nopsonsimple}, we need to show that the $nops$ function is equidistributed on $S_m(n)$. Now, for each $w\in\{0,1,\ldots, m-1\}$, let 
$$A_w=\{\ell\in S_m(n)\mid \ell_i=w\}.$$
Then, it is clear that $\{A_w\mid w\in\{0,1,\ldots, m-1\}$ forms a set partition of $S_m(n)$. Furthermore, since all the $m$-ary partitions in $A_w$ are simple, there is a bijection $g_{w,w'}:A_{w}\to A_{w'}$ given by 
$$g_{w,w'}((\ell_0,\ell_1,\ldots, w,\ldots,\ell_k)):= (\ell_0 + (w-w')\cdot m^i,\ell_1,\ldots, w',\ldots,\ell_k)$$ 
so that $|A_w|=|A_{w'}|$ for all $w,w'\in \{0,1,\ldots, m-1\}$. Finally, let $\ell\in A_{0}$. Then for each $w\in\{0,1,\ldots,m-1\}$ we have $nops(g_{0,w}(\ell))\equiv nops(\ell) + w \pmod{m}$. Thus the $nops$ function is equidistributed mod $m$ on $S_m(n)$.

Conversely, suppose that $m-1\not\in\{n_1,\ldots,n_k\}$. First, assume that the only nonzero base-$m$ digits are $n_0$ and $n_k$ so that  by assumption $n_k\leq m-2$. Then, there are only $n_k+1\leq m-1$ simple partitions, and so the $nops$ function cannot be equidistributed mod $m$ on $S_m(n)$. Next, assume that $0<n_j\leq m-2$ for some $1\leq j<k$. Similar to the previous paragraph, for each $w\in\{0,1,\ldots, n_j\}$, let 
$$A_w=\{\ell\in S_m(n)\mid \ell_j=w\}.$$
As before, $|A_w|=|A_{w'}|$ for all $w,w'\in\{0,1,\ldots, n_j\}$ and for each $\ell\in A_0$ and each $w\in\{0,1,\ldots,n_j\}$ we have $nops(g_{0,w}(\ell))\equiv nops(\ell) + w \pmod{m}$. Since $n_j\leq m-2$, then the $nops$ function will be equidistributed mod $m$ on $S_m(n)$ if and only if the $nops$ function is equidistributed mod $m$ on $A_0$. However, we see that there is a bijection $h:A_0\to S_m(n-n_j\cdot m^j)$ given by
$$h((\ell_0,\ell_1,\ldots, 0, \ldots, \ell_k)) := (\ell_0-n_j\cdot m^j, \ell_1,\ldots, 0, \ldots, \ell_k).$$
Moreover, we note that $nops(h(\ell))\equiv nops(\ell) \pmod{m}$ so that $nops$ is equidistributed mod $m$ on $A_0$ if and only if $nops$ is equidistributed mod $m$ on $S_m(n-n_j\cdot m^j)$, which implies that $nops$ is equidistributed mod $m$ on $S_m(n)$ if and only if $nops$ is equidistributed mod $m$ on $S_m(n-n_j\cdot m^j)$. Since the digit sets of $n$ and $n-n_j\cdot m^j$ are identical except in position $j$, we can use this argument to deduce that $nops$ is equidistributed mod $m$ on $S_m(n)$ if and only if $nops$ is equidistributed mod $m$ on $S_m\left(n-\sum_{i=1}^{k-1}n_i\cdot m^i\right)$. However, $n-\sum_{i=1}^{k-1}n_i\cdot m^i=(n_0,0,\ldots, 0, n_k)$ and $n_k\leq m-2$; in this case, we have already shown that $nops$ is not equidistributed mod $m$ $S_m\left(n-\sum_{i=1}^{k-1}n_i\cdot m^i\right)$. The result follows.
 \end{proof}

\section{Detailed Example}
\label{partitionexample}

We illustrate the results of the previous two sections with an example. Let $m=3$ and consider $n=60=(0, 2, 0, 2)_3$. Then  the total number of $3$-ary partitions of $60$ is $117$, i.e. $b_3(60)=117$. Of these $117$, there are $9$ simple partitions listed in the box below.

\vspace{.1in}
\tikzstyle{mybox} = [draw=black, fill=white, very thick,
    rectangle, rounded corners, inner sep=10pt, inner ysep=20pt,scale=.9]
\tikzstyle{fancytitle} =[draw=black,fill=white, text=black]

 \begin{tikzpicture}
 \node [mybox] (simple){
 \begin{minipage}{6in}
[0, 2, 0, 2],
 [3, 1, 0, 2],
 [6, 0, 0, 2],
 [27, 2, 0, 1],
 [30, 1, 0, 1],
 [33, 0, 0, 1],
 [54, 2, 0, 0],
 [57, 1, 0, 0],
 [60, 0, 0, 0]
 \end{minipage}
 };
 \node[fancytitle, right=10pt] at (simple.north west) {$S_3(60)$};
 \end{tikzpicture}

\vspace{.1in}

In the next two pages, we list the remaining $108$ non-simple partitions, those in $N_3(60)$, using the results in Section \ref{technical}. The numbers $3$-dominated by 60 are 
$$0, 3, 6, 27, 30, 33, 54, 57, 60.$$ 
Let $f$ represent $f_{3,60}$. It turns out that $f^{-1}(54)$, $f^{-1}(57)$, and $f^{-1}(60)$ are all empty. There are 6 partitions in $f^{-1}(0)$ and $f^{-1}(3)$; there are 69 partitions in $f^{-1}(6)$; there are 3 partitions in $f^{-1}(27)$ and $f^{-1}(30)$; and there are 21 partitions in $f^{-1}(33)$. All of the nonempty inverse images are listed below; the subsets correspond to the nonempty sets $B(z)$ for $1\leq z\leq 3$ and then the subsets of $B(z)$ correspond to the partition given by $\simeq_{b,z}$ guaranteed by Lemma \ref{simeqpartition}. The most representative example is that of $f^{-1}(6)$ as it contains both $B(1)$ and $B(2)$ ($B(3)=\emptyset$) and $B(1)$ is further partitioned into six equivalence classes for $\simeq_{6,1}$.

We can then check that each the cardinality of the equivalence classes of $\simeq_{b,z}$ is a multiple of $3$ and the $nops$ function is equidistributed mod $3$ on these smallest parts (see the proof of Theorem \ref{equidistributethm}) thus showing that the $nops$ function is equidistributed on $N_3(60)$. 

\vfill

\newpage

\begin{tikzpicture}
\node [mybox] (inverse0){%
    \begin{minipage}{6in}
      \begin{tikzpicture}
 	\node [mybox] (B2){
 	\begin{minipage}{5.5in}
	 \vspace{.1in}

 	\begin{tikzpicture}
 	\node [mybox] (zz){
 \begin{minipage}{5.2in}
[6, 0, 6, 0], [15, 0, 5, 0], [24, 0, 4, 0] ,[33, 0, 3, 0], [42, 0, 2, 0], [51, 0, 1, 0]
 \end{minipage}
 };
 \node[fancytitle, right=10pt, scale=.45] at (zz.north west) {$[*,*,*,0]$};
 \end{tikzpicture}
\end{minipage}
};
\node[fancytitle, right=10pt, scale=.65] at (B2.north west) {$B(2)$};
\end{tikzpicture}

\end{minipage}
};
\node[fancytitle, right=10pt, text=blue] at (inverse0.north west) {$f^{-1}(0)$; $0=(0,0,0,0)_3$};

\end{tikzpicture}%

\vspace{.15in}

\begin{tikzpicture}
\node [mybox] (inverse3){%
    \begin{minipage}{6in}
      \begin{tikzpicture}
 	\node [mybox] (B2){
 	\begin{minipage}{5.5in}
	 \vspace{.1in}

 	\begin{tikzpicture}
 	\node [mybox] (zz){
 \begin{minipage}{5.2in}
[3, 1, 6, 0], [12, 1, 5, 0] ,[21, 1, 4, 0], [30, 1, 3, 0], [39, 1, 2, 0], [48, 1, 1, 0]

 \end{minipage}
 };
 \node[fancytitle, right=10pt, scale=.45] at (zz.north west) {$[*,*,*,0]$};
 \end{tikzpicture}
\end{minipage}
};
\node[fancytitle, right=10pt, scale=.65] at (B2.north west) {$B(2)$};
\end{tikzpicture}

\end{minipage}
};
\node[fancytitle, right=10pt, text=red] at (inverse0.north west) {$f^{-1}(3)$; $3=(0,1,0,0)_3$};

\end{tikzpicture}%

\vspace{.15in}

\begin{tikzpicture}
\node [mybox] (inverse27){%
    \begin{minipage}{6in}
      \begin{tikzpicture}
 	\node [mybox] (B2){
 	\begin{minipage}{5.5in}
	 \vspace{.1in}

 	\begin{tikzpicture}
 	\node [mybox] (zz){
 \begin{minipage}{5.2in}
[6, 0, 3, 1], [15, 0, 2, 1], [24, 0, 1, 1]
 \end{minipage}
 };
 \node[fancytitle, right=10pt, scale=.45] at (zz.north west) {$[*,*,*,1]$};
 \end{tikzpicture}
\end{minipage}
};
\node[fancytitle, right=10pt, scale=.65] at (B2.north west) {$B(2)$};
\end{tikzpicture}

\end{minipage}
};
\node[fancytitle, right=10pt,text=OliveGreen] at (inverse27.north west) {$f^{-1}(27)$; $27=(0,0,0,1)_3$};

\end{tikzpicture}%

\vspace{.3in}

\begin{tikzpicture}
\node [mybox] (inverse30){%
    \begin{minipage}{6in}
      \begin{tikzpicture}
 	\node [mybox] (B2){
 	\begin{minipage}{5.5in}
	 \vspace{.1in}
 	\begin{tikzpicture}
 	\node [mybox] (zz){
 \begin{minipage}{5.2in}
[3, 1, 3, 1], [12, 1, 2, 1], [21, 1, 1, 1]
 \end{minipage}
 };
 \node[fancytitle, right=10pt, scale=.45] at (zz.north west) {$[*,*,*,1]$};
 \end{tikzpicture}
\end{minipage}
};
\node[fancytitle, right=10pt, scale=.65] at (B2.north west) {$B(2)$};
\end{tikzpicture}

\end{minipage}
};
\node[fancytitle, right=10pt,text=cyan] at (inverse30.north west) {$f^{-1}(30)$; $30=(0,1,0,1)_3$};

\end{tikzpicture}%

\vspace{.3in}

\begin{tikzpicture}
\node [mybox] (inverse6){%
    \begin{minipage}{6in}
     
 \begin{tikzpicture}
 \node [mybox] (B1){
 \begin{minipage}{5.5in}
  \vspace{.1in}

 \begin{tikzpicture}
 \node [mybox] (zz){
 \begin{minipage}{5in}
[0, 20, 0, 0],
[3, 19, 0, 0],
[6, 18, 0, 0],
[9, 17, 0, 0],
[12, 16, 0, 0],\newline
[15, 15, 0, 0],
[18, 14, 0, 0],
[21, 13, 0, 0],
[24, 12, 0, 0],
[27, 11, 0, 0],
[30, 10, 0, 0],
[33, 9, 0, 0],
[36, 8, 0, 0],
[39, 7, 0, 0],
[42, 6, 0, 0],\newline
[45, 5, 0, 0],
[48, 4, 0, 0],
[51, 3, 0, 0]
 \end{minipage}
 };
 \node[fancytitle, right=10pt, scale=.45] at (zz.north west) {$[*,*,0,0]$};
 \end{tikzpicture}
 
 \vspace{.1in}

 \begin{tikzpicture}
 \node [mybox] (oz){
 \begin{minipage}{5in}
 [0, 17, 1, 0],
[3, 16, 1, 0],
[6, 15, 1, 0],
[9, 14, 1, 0],
[12, 13, 1, 0],\newline
[15, 12, 1, 0],
[18, 11, 1, 0],
[21, 10, 1, 0],
[24, 9, 1, 0],
[27, 8, 1, 0],\newline
[30, 7, 1, 0],
[33, 6, 1, 0],
[36, 5, 1, 0],
[39, 4, 1, 0],
[42, 3, 1, 0]

 \end{minipage}
 };
 \node[fancytitle, right=10pt, scale=.45] at (oz.north west) {$[*,*,1,0]$};
 \end{tikzpicture}

 \vspace{.1in}

 \begin{tikzpicture}
 \node [mybox] (tz){
 \begin{minipage}{5in}
[0, 14, 2, 0],
[3, 13, 2, 0],
[6, 12, 2, 0],
[9, 11, 2, 0],
[12, 10, 2, 0],\newline
[15, 9, 2, 0],
[18, 8, 2, 0],
[21, 7, 2, 0],
[24, 6, 2, 0],
[27, 5, 2, 0],\newline
[30, 4, 2, 0],
[33, 3, 2, 0]
 \end{minipage}
 };
 \node[fancytitle, right=10pt, scale=.45] at (tz.north west) {$[*,*,2,0]$};
 \end{tikzpicture}

 \vspace{.1in}

 \begin{tikzpicture}
 \node [mybox] (thz){
 \begin{minipage}{5in}
[0, 11, 3, 0],
[3, 10, 3, 0],
[6, 9, 3, 0],
[9, 8, 3, 0],
[12, 7, 3, 0],\newline
[15, 6, 3, 0],
[18, 5, 3, 0],
[21, 4, 3, 0],
[24, 3, 3, 0]

 \end{minipage}
 };
 \node[fancytitle, right=10pt, scale=.45] at (thz.north west) {$[*,*,3,0]$};
 \end{tikzpicture}

 \vspace{.1in}

 \begin{tikzpicture}
 \node [mybox] (fz){
 \begin{minipage}{5in}
[0, 8, 4, 0],
[3, 7, 4, 0],
[6, 6, 4, 0],
[9, 5, 4, 0],
[12, 4, 4, 0],
[15, 3, 4, 0]
 \end{minipage}
 };
 \node[fancytitle, right=10pt, scale=.45] at (fz.north west) {$[*,*,4,0]$};
 \end{tikzpicture}

 \vspace{.1in}

 \begin{tikzpicture}
 \node [mybox] (fiz){
 \begin{minipage}{5in}
[0, 5, 5, 0],
[3, 4, 5, 0],
[6, 3, 5, 0]
 \end{minipage}
 };
 \node[fancytitle, right=10pt, scale=.45] at (fiz.north west) {$[*,*,5,0]$};
 \end{tikzpicture}

 \end{minipage}
 };
 \node[fancytitle, right=10pt, scale=.65] at (B1.north west) {$B(1)$};
 \end{tikzpicture}
 
  \vspace{.1in}

  \begin{tikzpicture}
 \node [mybox] (B2){
 \begin{minipage}{5.5in}
  \vspace{.1in}

  	\begin{tikzpicture}
 	\node [mybox] (zz){
 \begin{minipage}{5in}
$[0, 2, 6, 0]$, $[9, 2, 5, 0]$, $[18, 2, 4, 0]$, $[27, 2, 3, 0]$, $[36, 2, 2, 0]$, $[45, 2, 1, 0]$ \end{minipage}
 };
 \node[fancytitle, right=10pt, scale=.45] at (zz.north west) {$[*,*,*,0]$};
 \end{tikzpicture}
 \end{minipage}
 };
 \node[fancytitle, right=10pt, scale=.65] at (B2.north west) {$B(2)$};
 \end{tikzpicture}

     \end{minipage}
};
\node[fancytitle, right=10pt,text=magenta] at (inverse6.north west) {$f^{-1}(6)$; $6=(0,2,0,0)_3$};

\end{tikzpicture}%

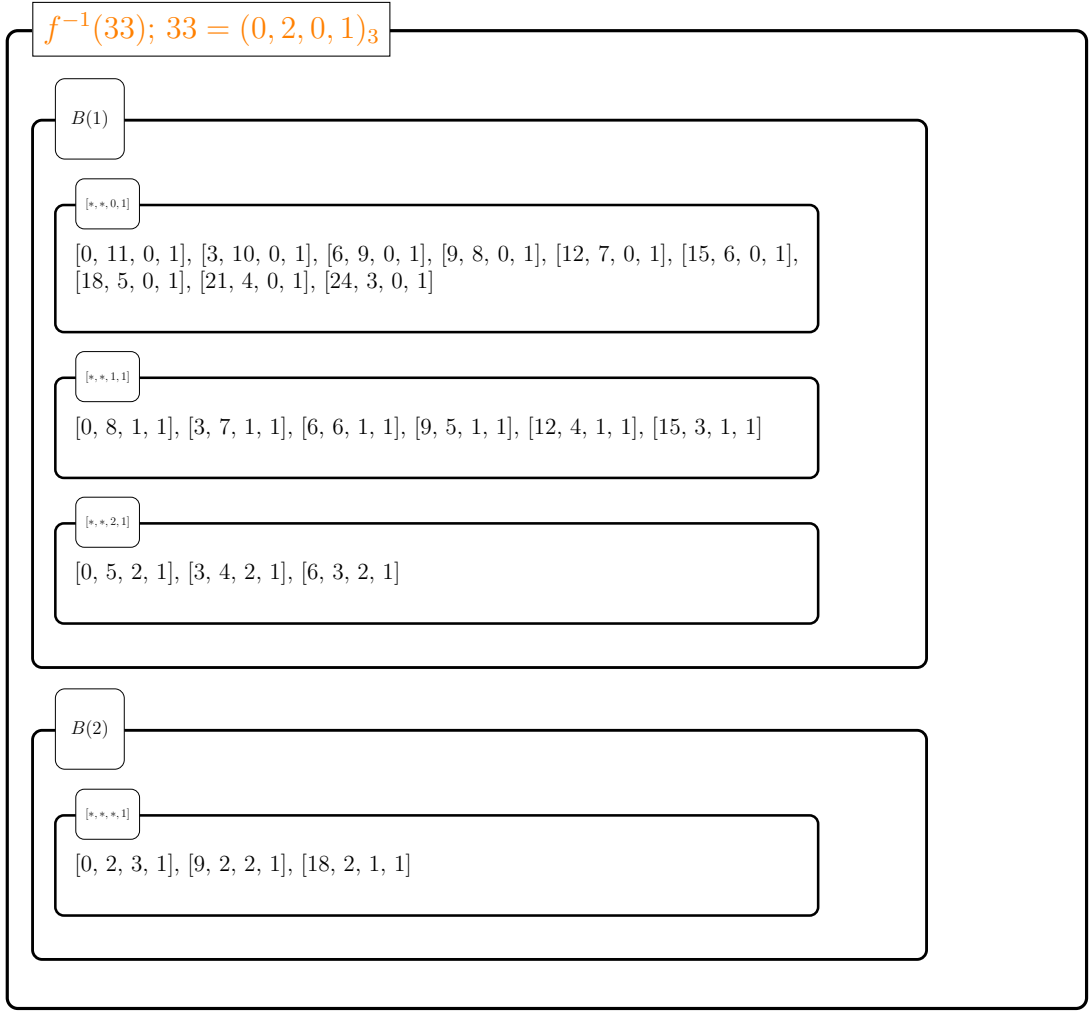
\begin{tikzpicture}
\node [mybox] (inverse33){%
    \begin{minipage}{6in}

      \begin{tikzpicture}
 	\node [mybox] (B1){
 	\begin{minipage}{5.5in}
	 \vspace{.1in}

 	\begin{tikzpicture}
 	\node [mybox] (zo){
 \begin{minipage}{5.2in}

[0, 11, 0, 1],
[3, 10, 0, 1],
[6, 9, 0, 1],
[9, 8, 0, 1],
[12, 7, 0, 1],
[15, 6, 0, 1],
[18, 5, 0, 1],
[21, 4, 0, 1],
[24, 3, 0, 1]

 \end{minipage}
 };
 \node[fancytitle, right=10pt, scale=.45] at (zo.north west) {$[*,*,0,1]$};
 \end{tikzpicture}

 \vspace{.1in}

 	\begin{tikzpicture}
 	\node [mybox] (oo){
 \begin{minipage}{5.2in}

[0, 8, 1, 1],
[3, 7, 1, 1],
[6, 6, 1, 1],
[9, 5, 1, 1],
[12, 4, 1, 1],
[15, 3, 1, 1]

 \end{minipage}
 };
 \node[fancytitle, right=10pt, scale=.45] at (oo.north west) {$[*,*,1,1]$};
 \end{tikzpicture}

 \vspace{.1in}

	\begin{tikzpicture}
 	\node [mybox] (to){
 \begin{minipage}{5.2in}
[0, 5, 2, 1],
[3, 4, 2, 1],
[6, 3, 2, 1]

 \end{minipage}
 };
 \node[fancytitle, right=10pt, scale=.45] at (to.north west) {$[*,*,2,1]$};
 \end{tikzpicture}

\end{minipage}
};
\node[fancytitle, right=10pt, scale=.65] at (B1.north west) {$B(1)$};
\end{tikzpicture}
 \vspace{.1in}

      \begin{tikzpicture}
 	\node [mybox] (B2){
 	\begin{minipage}{5.5in}
	 \vspace{.1in}

 	\begin{tikzpicture}
 	\node [mybox] (zz){
 \begin{minipage}{5.2in}
[0, 2, 3, 1],
[9, 2, 2, 1],
[18, 2, 1, 1]

 \end{minipage}
 };
 \node[fancytitle, right=10pt, scale=.45] at (zz.north west) {$[*,*,*,1]$};
 \end{tikzpicture}
\end{minipage}
};
\node[fancytitle, right=10pt, scale=.65] at (B2.north west) {$B(2)$};
\end{tikzpicture}

\end{minipage}
};
\node[fancytitle, right=10pt, text=orange] at (inverse33.north west) {$f^{-1}(33)$; $33=(0,2,0,1)_3$};

\end{tikzpicture}%

\section{Extensions}

In this section, we briefly discuss a possible way to extend our results to other congruence relations. We note that the set of non-simple $m$-ary partitions $N_m(n)$ can be defined as
$$N_m(n) = \{\ell\in P_m(n)\mid \ell_j>n_j\mbox{ for some $j\geq 1$}\}$$
where $n=(n_0,\ldots, n_k)_m$ is the base-$m$ representation of $n$. Consider the following generalizations. For any $c\geq 1$, we let 
$$N_{m,c} = \{\ell\in P_m(n)\mid \ell_j>n_j, \ell_{j+1}=n_{j+1},\ldots, \ell_{j+c}=n_{j+c} \mbox{ for some $j\geq 1$}\},$$
where we note that $N_m(n)$ can be interpreted as $N_{m,0}$. Then, we can prove a result analogous to Lemma \ref{maxchain} that shows $|N_{m,c}|\equiv 0 \pmod{m^{c+1}}$. Therefore, if we can determine the size of the set 
$$S_{m,c}(n):=P_m(n)\setminus N_{m,c}(n)$$ 
using only knowledge of $n$ (possibly the base-$m$ representation of $n$), then we will obtain interesting congruence properties for $b_m(n)$ mod $m^{c+1}$ for any $c$.

\label{future}

\nocite{andrewsbook}

\bibliographystyle{plain}
\bibliography{equidistribution_m-ary.bib}

\end{document}